\documentclass[12pt, reqno]{amsart}

\usepackage[margin=1in]{geometry}

\usepackage{amsfonts} 
\usepackage{amssymb}
\usepackage{amsmath}
\usepackage{amsthm}
\usepackage{latexsym}
\usepackage{graphicx}
\usepackage{xypic}
\usepackage{hyperref}
\hypersetup{colorlinks,citecolor=red,pdfstartview=FitH, pdfpagemode=None}

\numberwithin{equation}{section}

\newtheorem{theorem}{Theorem}[section]
\newtheorem{lemma}[theorem]{Lemma}

\theoremstyle{definition}

\newtheorem{remark}[theorem]{Remark}

\renewcommand{\ge}{\geqslant}
\renewcommand{\le}{\leqslant}

\def\ui{\|\hspace{-.25mm} |\,}

\def\C{\mathbb C}

\def\P{\mathbb P}

\def\Cnn{\C_{n\times n}}

\begin{document}

\title[An inequality for $t$-geometric means]{An inequality for $t$-geometric means}


\author{Trung Hoa Dinh$^a$}
\address{Division of Computational Mathematics and Engineering, Institute for Computational Science, Ton Duc Thang University, Ho Chi Minh City, Vietnam\\
Faculty of Civil Engineering, Ton Duc Thang University, Vietnam\\
Department of Mathematics and Statistics,  Auburn University, Auburn, AL 36849,  USA\\}
\email{dinhtrunghoa@tdt.edu.vn, trunghoa.math@gmail.com}

\footnote{$^A$Research of the first author is funded by Vietnam National Foundation for Science and Technology Development (NAFOSTED) under grant number 101.04-2014.40.}

\date{\today}
\keywords{$t$-geometric mean, positive definite matrices, log-majorization, unitarily invariant norms}
\subjclass[2010]{15A45, 
15B48, 
53C35
}

\begin{abstract}
In this note we prove an inequality for $t$-geometric means that immediately implies the recent results of Audenaert \cite{AU} and Hayajneh-Kittaneh \cite{HK}. 
\end{abstract}

\maketitle

\section{Introduction}

Recall that a norm $\ui\cdot \ui$ on the algebra $\Cnn$ of $n\times n$ complex matrices is unitarily invariant if $\ui UAV \ui = \ui A \ui $ for any unitary matrices $U, V \in \Cnn$ and any $A \in \Cnn$. 

In \cite{AU} Audenaert proved that  if $A_i, B_i \in \P_n$ $(i=1, \cdots, m)$,
 such that $A_iB_i = B_iA_i$, then for any unitarily invariant norm $\ui\cdot \ui$ on $\Cnn$,
\begin{equation}\label{AU}
\ui  \sum_{i=1}^m A_iB_i \ui  \le \ui (\sum_{i=1}^m A_i^{1/2}B_i^{1/2})^2 \ui \le \ui  (\sum_{i=1}^m A_i) (\sum_{i=1}^m B_i)\ui.
\end{equation}
{In particular, this result confirms a conjecture of Hayajneh and Kittaneh in \cite{HK} and answers a question of Bourin. Very recently Lin \cite{linlaa} gave another proof of inequality (\ref{AU}). In the next section, based on a result of  Bourin and Uchiyama in \cite{BU} we prove an inequality for $t$-geometric means that immediately implies (\ref{AU}). 
\section{Main result}

The following lemma is follows from the well-known Hiai-Ando log-majorization Theorem and \cite[Proposition 2.2]{hoatam}.
\begin{lemma}\label{uia}\rm
Let $A, B \in \P_n$ and   $t \in [0, 1]$ and $s > 0$. For all unitarily invariant norms $\ui\cdot\ui$ on $\Cnn$,
\begin{equation*}
||| (A\sharp_t B)^r|||  \le ||| A^r \sharp_t B^r||| \le  \ui \left( B^{rts/2}A^{(1-t)rs}B^{rts/2} \right)^{1/s}\ui \le \ui \left( A^{(1-t)rs}B^{rts} \right)^{1/s}\ui.
\end{equation*}
\end{lemma}
The following theorem was obtained by Bourin and Uchiyama \cite[Theorem 1.2]{BU}.  
\begin{theorem}\label{theorBU}
Let $A_i \ge 0$ $(i= 1, \cdots, m)$. Then for every non-negative convex function $f$ on $[0, \infty)$ with $f(0) =0$ and for any unitarily invariant norm $|||\cdot |||$
\begin{align}\label{BU}
||| f(A_1)+ f(A_2) + \cdots f(A_m)||| \le ||| f(A_1+ A_2 + \cdots + A_m) |||.
\end{align}  
The above inequality is reversed if $f$ is non-negative concave function on $[0, \infty)$ with $f(0) =0$.
\end{theorem}
Now we will prove our main theorem.
\begin{theorem}\label{thm3}\rm
Let $A_i, B_i \in M_n^+$, $i=1, \dots, m$. Then for any unitarily invariant norm $\ui\cdot\ui$ on $M_n$ and $r \ge 1$ we have
\begin{align}\label{ineq3}
\ui\sum_{i=1}^m (A_i\sharp_t B_i)^r \ui \le \ui (\sum_{i=1}^m A_i)^{r/4}(\sum_{i=1}^m B_i)^{r/2} (\sum_{i=1}^m A_i)^{r/4} \ui \le \ui (\sum_{i=1}^m A_i)^{r/2}(\sum_{i=1}^m B_i)^{r/2} \ui.
\end{align}
\end{theorem}
\begin{proof}
Since the function $x^r$ is convex and monotone increasing, by Theorem \ref{theorBU} and the concavity of the $t$-geometric means, we have
\begin{align}\label{11}
\ui\sum_{i=1}^m(A_i\sharp_t B_i)^r \ui  \le \ui (\sum_{i=1}^m A_i \sharp_t B_i)^r \ui  \le \ui \big((\sum_{i=1}^m A_i) \sharp_t (\sum_{i=1}^m B_i)\big)^r \ui.
\end{align}
On account of Lemma \ref{uia} 
\begin{equation}\label{22}
\ui \big((\sum_{i=1}^m A_i) \sharp_t (\sum_{i=1}^m B_i)\big)^r \ui \le  \ui (\sum_{i=1}^m A_i)^{r/4}(\sum_{i=1}^m B_i)^{r/2} (\sum_{i=1}^m A_i)^{r/4} \ui \le \ui (\sum_{i=1}^m A_i)^{r/2}(\sum_{i=1}^m B_i)^{r/2} \ui.
\end{equation}
Combining (\ref{11}) and (\ref{22}), we get (\ref{ineq3}).
\end{proof}

\begin{remark}
The case when $r=2$ was considered in \cite{hoatam}. The results of Audenaert \cite{AU} and Hayajneh-Kittaneh \cite{HK} are the special case of (\ref{ineq3}) when $r=2$ and $t=\frac{1}{2}$.
\end{remark}

\end{document}